\documentclass[11pt]{article}
\usepackage[a4paper,margin=1.7cm]{geometry}

\usepackage{url,hyperref}
\usepackage{graphicx,amsfonts,amsmath,amsthm,amssymb}
\usepackage{color}
\usepackage[ruled]{algorithm2e}
\usepackage{amssymb}
\usepackage{adjustbox}
\usepackage{enumitem}
\usepackage{booktabs}       

\definecolor{darkgreen}{RGB}{0,160,0}
\definecolor{orange}{RGB}{255,165,0}
\definecolor{purple}{RGB}{147,112,219}

\newcommand{\E}{\mathbb{E}}
\newcommand{\N}{\mathbb{N}}

\newcommand{\1}{\mathbf{1}}

\newcommand{\Ber}{\mathrm{Ber}}

\newcommand{\TOP}{\operatorname{TOP}}

\newcommand{\KL}{\operatorname{KL}}

\renewcommand{\hat}{\widehat}
\renewcommand{\tilde}{\widetilde}

\renewcommand{\P}{\mathbb{P}}

\newtheorem{lemma}{Lemma}
\newtheorem{theorem}{Theorem}
\newtheorem{proposition}{Proposition}

\newtheorem{remark}{Remark}

\title{A KL-LUCB Bandit Algorithm for \\ Large-Scale Crowdsourcing }

\author{Bob Mankoff\\
Former Cartoon Editor of the New Yorker \\
\href{mailto:bmankoff@hearst.com}{bmankoff@hearst.com}
\and
Robert Nowak\\
University of Wisconsin -- Madison\\
\href{mailto:rdnowak@wisc.edu}{rdnowak@wisc.edu}
\and
Ervin T\'anczos\footnote{To whom correspondence should be addressed.}\\
University of Wisconsin -- Madison\\
\href{mailto:tanczos@wisc.edu}{tanczos@wisc.edu}
\thanks{This work was partially supported by the NSF grant IIS-1447449 and the AFSOR grant FA9550-13-1-0138.}
}

\begin{document}

\maketitle

\begin{abstract}
This paper focuses on best-arm identification in multi-armed bandits with bounded rewards. We develop an algorithm that is a fusion of lil-UCB and KL-LUCB, offering the best qualities of the two algorithms in one method. This is achieved by proving a novel anytime confidence bound for the mean of bounded distributions, which is the analogue of the LIL-type bounds recently developed for sub-Gaussian distributions. We corroborate our theoretical results with numerical experiments based on the New Yorker Cartoon Caption Contest.
\end{abstract}

\section{Multi-Armed Bandits for Large-Scale Crowdsourcing}

This paper develops a new multi-armed bandit (MAB) for large-scale
crowdsourcing, in the style of the KL-UCB
\cite{KLUCB_2011_Garivier,maillard2011finite,cappe2013kullback}. Our
work is strongly motivated by crowdsourcing
contests, like the New Yorker Cartoon Caption contest \cite{cnet}\footnote{For more details on the New Yorker Cartoon Caption Contest, see the Supplementary Materials.}.
The new approach targets the ``best-arm identification problem''
\cite{audibert2010best} in the fixed confidence setting and addresses
two key limitations of existing theory and algorithms:
\begin{enumerate}[leftmargin=*]
 \item[(i)] State of the art algorithms for best arm identification are
  based on sub-Gaussian confidence bounds \cite{lilUCB_2014} and fail
  to exploit the fact that rewards are usually bounded in
  crowdsourcing applications.
\item[(ii)] Existing KL-UCB algorithms for best-arm identification do
  exploit bounded rewards \cite{KLUCB_2013_kaufmann} , but have
  suboptimal performance guarantees in the fixed confidence setting,
  both in terms of dependence on problem-dependent hardness parameters
  (Chernoff information) and on the number of arms, which can be large
  in crowdsourcing applications.
\end{enumerate}
The new algorithm we propose and analyze is called {\bf lil-KLUCB}, since it
is inspired by the lil-UCB algorithm \cite{lilUCB_2014} and the KL-LUCB
algorithm \cite{KLUCB_2013_kaufmann}.  The lil-UCB algorithm is based
on sub-Gaussian bounds and has a sample complexity for best-arm
identification that scales as
$$\sum_{i\geq 2} \Delta_i^{-2} \log (\delta^{-1}\log \Delta_i^{-2}) \
, $$ where $\delta \in (0,1)$ is the desired confidence and
$\Delta_i = \mu_1 -\mu_i$ is the gap between the means of the best arm
(denoted as arm $1$) and arm~$i$.  If the rewards are in $[0,1]$, then the
KL-LUCB algorithm has a sample complexity scaling essentially like\footnote{A
  more precise characterization of the sample complexity is given in
  Section~2.}
$$\sum_{i\geq 2} (D_i^*)^{-1} \log (n \delta^{-1} (D^*_i)^{-1}) \ , $$ where $n$
is the number of arms and $D_i^* := D^* (\mu_1 ,\mu_i)$ is the
Chernoff-information between a $\Ber (\mu_1 )$ and a $\Ber (\mu_i )$
random variable\footnote{The Chernoff-information between random
  variables $\Ber (x)$ and $\Ber (y)$ ($0<x<y<1$) is
  $D^* (x,y) = D(z^*,x) = D(z^*,y)$, where $D(z,x)=z\log \tfrac{z}{x}+(1-z)\log \tfrac{1-z}{1-x}$ and $z^*$ is the unique
  $z\in (x,y)$ such that $D(z,x)=D(z,y)$.}. Ignoring the logarithmic
factor, this bound is optimal for the case of Bernoulli rewards
\cite{kaufmann2016complexity,Kevin_2017}. Comparing these two bounds,
we observe that KL-LUCB may offer benefits since
$D_i^* =D^* (\mu_1 ,\mu_i) \geq (\mu_1 -\mu_i)^2 /2 = \Delta_i^2/2$, but lil-UCB has
better logarithmic dependence on the $\Delta_i^2$ and no explicit
dependence on the number of arms $n$. Our new algorithm lil-KLUCB
offers the best of both worlds, providing a sample complexity that
scales essentially like
$$\sum_{i\geq 2} (D_i^*)^{-1} \log (\delta^{-1} \log (D^*_i)^{-1}) \
. $$ The key to this result is a novel anytime confidence bound for
sums of bounded random variables, which requires a significant departure from previous analyses of KL-based
confidence bounds.  

The practical benefit of lil-KLUCB is illustrated in terms of the New
Yorker Caption Contest problem \cite{cnet}.  The goal of that
crowdsourcing task is to identify the funniest cartoon caption from a
batch of $n\approx 5000$ captions submitted to the contest each
week. The crowd provides ``3-star'' ratings for the captions, which
can be mapped to $\{0,1/2,1\}$, for example.  Unfortunately, many of
the captions are not funny, getting average ratings close to $0$ (and
consequently very small variances).  This fact, however, is ideal for
KL-based confidence intervals, which are significantly tighter than
those based on sub-Gaussianity and the worst-case variance of $1/4$.
Compared to existing methods, the lil-KLUCB algorithm better addresses
the two key features in this sort of application: (1) a very large
number of arms, and (2) bounded reward distributions which, in many
cases, have very low variance.  In certain instances, this can have a
profound effect on sample complexity (e.g., $O(n^2)$ complexity for
algorithms using sub-Gaussian bounds vs.\ $O(n\log n)$ for lil-KLUCB,
as shown in Table~\ref{tab:KL_vs_SG}).

The paper is organized as follows. Section~\ref{sec:main}
defines the best-arm identification problem, gives the
lil-KLUCB algorithm and states the main results. We also briefly review
related literature, and compare the performance of lil-KLUCB to that of previous algorithms.
Section~\ref{sec:anytime} provides the main technical contribution
of the paper, a novel anytime confidence bound for
sums of bounded random variables. Section~\ref{sec:lil-KLUCB}
analyzes the performance of the lil-KLUCB algorithm.
Section~\ref{sec:experiment} provides experimental support for the
lil-KLUCB algorithm using data from the New Yorker
Caption Contest.


\section{Problem Statement and Main Results}\label{sec:main}

Consider a MAB problem with $n$ arms.  We use the shorthand notation
$[n]:=\{1,\dots,n\}$. For every $i\in [n]$ let
$\{ X_{i,j} \}_{j\in \N}$ denote the reward sequence of arm~$i$, and
suppose that $\P (X_{i,j} \in [0,1])=1$ for all $i\in [n],\ j\in
\N$. Furthermore, assume that all rewards are independent, and that
$X_{i,j} \sim \P_i$ for all $j\in \N$. Let the mean reward of arm~$i$
be denoted by $\mu_i$ and assume w.l.o.g. that
$\mu_1>\mu_2 \geq \dots \geq \mu_n$.

We focus on the best-arm identification problem in the
fixed-confidence setting. At every time $t\in \N$ we are allowed to
select an arm to sample (based on past rewards) and observe the next
element in its reward sequence. Based on the observed rewards, we wish
to find the arm with the highest mean reward. In the fixed confidence
setting, we prescribe a probability of error $\delta \in (0,1)$ and
our goal is to construct an algorithm that finds the best arm with
probability at least $1-\delta$.  Among $1-\delta$ accurate
algorithms, one naturally favors those that require fewer
samples. Hence proving upper bounds on the sample complexity of a
candidate algorithm is of prime importance.

The lil-KLUCB algorithm that we propose is a fusion of
lil-UCB \cite{lilUCB_2014} and KL-LUCB \cite{KLUCB_2013_kaufmann}, and
its operation is essentially a special instance of
LUCB++ \cite{Kevin_2017}.  At each time step $t$, let $T_i(t)$ denote
the total number of samples drawn from arm $i$ so far, and let
$\hat{\mu}_{i,T_i(t)}$ denote corresponding empirical mean. The algorithm
is based on lower and upper confidence bounds of the following general
form: for each $i\in [n]$ and any $\epsilon \in (0,1)$
\begin{align*}
L_i(t,\epsilon ) & = \inf \left\{ m<\hat{\mu}_{i,T_i(t)} :\ D \left( \hat{\mu}_{i,T_i(t)},m \right) \leq \frac{c \log \left( \kappa \log_2 (2 T_i(t)) /\epsilon \right)}{T_i(t)} \right\}
 \\
U_i(t,\epsilon ) & = \sup \left\{ m>\hat{\mu}_{i,T_i(t)} :\ D \left( \hat{\mu}_{i,T_i(t)},m \right) \leq \frac{c \log \left( \kappa \log_2 (2 T_i(t)) /\epsilon \right)}{T_i(t)} \right\}
 \
\end{align*}
where $c$ and $\kappa$ are small constants (defined in the next
section).  These bounds are designed so that with probability at
least $1-\epsilon$, $L_i (T_i(t),\epsilon ) \ \leq \ \mu_i \ \leq \ U_i (T_i(t),\epsilon
)$ holds for all $t\in \N$. For any $t\in \N$ let $\TOP (t)$ be the index of the arm with
the highest empirical mean, breaking ties at random. With this
notation, we state
the lil-KLUCB algorithm and our main theoretical result. \ \\

\begin{figure}[h]
\fbox{\parbox[b]{4.5in}{{\underline{\bf lil-KLUCB}}  \\ \vspace{-.25in}
    \\
\begin{enumerate}
\item{\textbf{Initialize} by sampling every arm once.}
\item{\textbf{While} $L_{\TOP (t)} (T_{\TOP (t)}(t),\delta /(n-1)) \leq \displaystyle{\max_{i\neq \TOP (t)}} U_i (T_i(t),\delta )$ \textbf{do:}
\begin{itemize}
\item{Sample the following two arms:
\begin{itemize}
\item{$\TOP (t)$, and}
\item{$\arg \displaystyle{\max_{i \neq \TOP (t)}} U_i (T_i(t),\delta )$}
\end{itemize}
and update means and confidence bounds.}
\end{itemize}
}
\item{\textbf{Output} $\TOP (t)$}
\end{enumerate}
}}
\end{figure}

\begin{theorem}\label{thm:samplecomp}
  For every $i\geq 2$ let $\tilde{\mu}_i \in (\mu_i ,\mu_1 )$, and
  $\tilde{\mu}=\max_{i\geq 2} \tilde{\mu}_i$. With probability at
  least $1-2\delta$, lil-KLUCB returns the arm with the largest mean
  and the total number of samples it collects is upper bounded by
\[
\inf_{\tilde{\mu}_2,\dots ,\tilde{\mu}_n} \frac{c_0 \log \left( (n-1) \delta^{-1} \log D^* (\mu_1 ,\tilde{\mu})^{-1} \right)}{D^* (\mu_1 ,\tilde{\mu})} + \sum_{i\geq 2} \frac{c_0 \log \left( \delta^{-1} \log D^* (\mu_i ,\tilde{\mu}_i)^{-1} \right)}{D^* (\mu_i ,\tilde{\mu}_i )} \ ,
\]
where $c_0$ is some universal constant, $D^*(x,y)$ is the Chernoff-information.
\end{theorem}

\begin{remark}\label{rem:best-arm_vs_top-k}
Note that the LUCB++ algorithm of \cite{Kevin_2017} is general enough to handle identification of the top $k$ arms (not just the best-arm). All arguments presented in this paper also go through when considering the top-$k$ problem for $k>1$. However, to keep the arguments clear and concise, we chose to focus on the best-arm problem only.
\end{remark}

\subsection{Comparison with previous work}\label{sec:comparison}

We now compare the sample complexity of lil-KLUCB to that
of the two most closely related algorithms, KL-LUCB
\cite{KLUCB_2013_kaufmann} and lil-UCB \cite{lilUCB_2014}. For a
detailed review of the history of MAB problems and the use of KL-confidence intervals for bounded rewards, we refer the
reader to \cite{cappe2013kullback, maillard2011finite,
  KLUCB_2011_Garivier}.

For the KL-LUCB algorithm, Theorem~3 of \cite{KLUCB_2013_kaufmann}
guarantees a high-probability sample complexity upper bound scaling as
\[
\inf_{c\in (\mu_1,\mu_2 )} \sum_{i\geq 1} (D^* (\mu_i ,c))^{-1} \log \left( n \delta^{-1} (D^* (\mu_i ,c))^{-1} \right) \ .
\]
Our result improves this in two ways. On one hand, we eliminate the
unnecessary logarithmic dependence on the number of arms $n$ in every
term. Note that the $\log n$ factor still appears in
Theorem~\ref{thm:samplecomp} in the term corresponding to the number
of samples on the best arm. It is shown in \cite{Kevin_2017} that this
factor is indeed unavoidable. The other improvement lil-KLUCB offers
over KL-LUCB is improved logarithmic dependence on the
Chernoff-information terms. This is due to the tighter confidence intervals derived in Section~\ref{sec:anytime}.

Comparing Theorem~\ref{thm:samplecomp} to the sample complexity of
lil-UCB, we see that the two are of the same form, the exception being
that the Chernoff-information terms take the place of the squared
mean-gaps (which arise due to the use of sub-Gaussian (SG) bounds). To
give a sense of the improvement this can provide, we compare the
sums\footnote{Consulting the proof of Theorem~\ref{thm:samplecomp} it
  is clear that the number of samples on the sub-optimal arms of
  lil-KLUCB scales essentially as $S_{\KL}$ w.h.p. (ignoring doubly logarithmic
  terms), and a similar argument can be made about lil-UCB. This
  justifies considering these sums in order to compare lil-KLUCB and
  lil-UCB.} 
\[
S_{\KL} = \sum_{i\geq 2} \frac{1}{D^* (\mu_i ,\mu_1)} \ \ \textrm{and} \
\ S_{\textrm{SG}} = \sum_{i\geq 2} \frac{1}{\Delta_i^2}  \  .
\]
Let $\mu ,\mu' \in (0,1),\ \mu < \mu'$ and $\Delta = |\mu - \mu'|$. Note that the Chernoff-information between $\Ber (\mu )$ and $\Ber (\mu' )$ can be expressed as
\[
D^* (\mu ,\mu' ) = \max_{x\in [\mu ,\mu' ]} \min \{ D(x,\mu ),D(x,\mu' ) \}  = D(x^*,\mu ) = D(x^* ,\mu') = \tfrac{D(x^* ,\mu )+D(x^* ,\mu' )}{2} \ ,
\]
for some unique $x^* \in [\mu ,\mu' ]$. Hence it follows that
$$D^* (\mu ,\mu' ) \ \geq \ \min_{x\in [\mu ,\mu' ]} \frac{D(x,\mu ) + D(x,\mu' )}{2} 
\ = \ \log \frac{1}{\sqrt{\mu (\mu +\Delta )} + \sqrt{(1-\mu )(1-\mu
    -\Delta )}} \ . $$ Using this with every term in $S_{\KL}$ gives
us an upper bound on that sum.  If the means are all bounded well away
from $0$ and $1$, then $S_{\KL}$ may not differ that much from
$S_{\textrm{SG}}$. There are some situations however, when the two
expressions behave radically differently. As an example, consider a
situation when $\mu_1 =1$. In this case we get
\[
S_{\KL} \leq \sum_{i\geq 2} \frac{2}{\log \tfrac{1}{1-\Delta_i}}  \leq  2 \sum_{i\geq 2} \frac{1}{\Delta_i} \ll \sum_{i\geq 2} \frac{1}{\Delta_i^2}  =  S_{\textrm{SG}} \ .
\]
Table~\ref{tab:KL_vs_SG} illustrates the difference between the scaling of the sums $S_{\KL}$ and $S_{\textrm{SG}}$ when the gaps have the parametric form $\Delta_i = (i/n)^\alpha$.

\begin{table}[h]
  \caption{$S_{\KL}$ versus $S_{\textrm{SG}}$ for mean gaps $\Delta_i
    = (\frac{i}{n})^\alpha, \ i=1,\dots,n$}
  \label{tab:KL_vs_SG}
  \centering
  \begin{tabular}{l|ccccc}
    \toprule
    $\alpha$ & $\in (0,1/2)$ & $1/2$ & $\in (1/2,1)$ & $1$ & $\in (1,\infty )$ \\
    \midrule
    $S_{\KL}$ & $n$ & $n$ & $n$ & $n \log n$ & $n^\alpha$ \\
    $S_{\textrm{SG}}$ & $n$ & $n \log n$ & $n^{2\alpha}$ & $n^2$ & $n^{2\alpha}$ \\
    \bottomrule
  \end{tabular}
\end{table}

We see that
KL-type confidence bounds can sometimes provide a significant
advantage in terms of the sample complexity. Intuitively, the gains
will be greatest when many of the means are close to 0 or 1 (and hence
have low variance). We will illustrate in Section~\ref{sec:experiment}
that such gains often also manifest in practical applications like the
New Yorker Caption Contest problem.


\section{Anytime Confidence Intervals for Sums of Bounded Random Variables}\label{sec:anytime}

The main step in our analysis is proving a sharp anytime
confidence bound for the mean of bounded random variables. These
will be used to show, in Section~\ref{sec:lil-KLUCB}, that lil-KLUCB
draws at most $O( (D_i^*)^{-1} \log \log (D^*_i)^{-1})$ samples from a
suboptimal arm $i$, where $D_i^* := D^* (\mu_1 ,\mu_i)$ is the
Chernoff-information between a $\Ber (\mu_1 )$ and a $\Ber (\mu_i )$
random variable and arm $1$ is the arm with the largest mean.
The iterated log factor is a necessary consequence of the
law-of-the-iterated logarithm \cite{lilUCB_2014}, and in it is in this
sense that we call the bound sharp. Prior work on MAB algorithms based
on KL-type confidence bounds
\cite{KLUCB_2011_Garivier,maillard2011finite,cappe2013kullback} did
not focus on deriving tight anytime confidence bounds. 

Consider a sequence of iid random variables $Y_1, Y_2, \dots$ that are bounded in $[0,1]$ and have mean $\mu$. Let $\hat{\mu}_t = \frac{1}{t}\sum_{j\in [t]} Y_j$ be the empirical mean of the observations up to time $t\in \N$. 

\begin{theorem}\label{thm:anytime_1}
Let $\mu \in [0,1]$ and $\delta \in (0,1)$ be arbitrary. Fix any $l\geq 0$ and set $N =2^l$, and define
\[
\kappa (N ) = \delta^{1/(N+1)} \left( \sum_{t\in [N]} \1_{\{ l\neq 0 \}} \log_2 (2t)^{-\frac{N+1}{N}} + N \sum_{k\geq l} (k+1)^{-\frac{N+1}{N}} \right)^{\frac{N}{N+1}} \ .
\]
\begin{itemize}
\item[(i)]{Define the sequence $z_t \in (0,1-\mu],\ t\in \N$ such that
\begin{equation}\label{eqn:def_z_t}
D \left( \mu + \tfrac{N}{N+1} z_t,\mu \right) = \frac{\log \left( \kappa (N ) \log_2 (2t) /\delta \right)}{t} \ ,
\end{equation}
if a solution exists, and $z_t =1-\mu$ otherwise. Then $\P \left( \exists t\in \N :\ \hat{\mu}_t - \mu > z_t \right) \leq \delta$.
}
\item[(ii)]{
Define the sequence $z_t >0,\ t\in \N$ such that
\[
D \left( \mu - \tfrac{N}{N+1} z_t,\mu \right) = \frac{\log \left( \kappa (N ) \log_2 (2t) /\delta \right)}{t} \ ,
\]
if a solution exists, and $z_t =\mu$ otherwise. Then $\P \left( \exists t\in \N :\ \hat{\mu}_t - \mu < -z_t \right) \leq \delta$.
}
\end{itemize}
\end{theorem}

The result above can be used to construct anytime confidence bounds for the mean as follows. Consider part \emph{(i)} of Theorem~\ref{thm:anytime_1} and fix $\mu$. The result gives a sequence $z_t$ that upper bounds the deviations of the empirical mean. It is defined through an equation of the form $D(\mu+ N z_t/(N+1),\mu)=f_t$. Note that the arguments of the function on the left must be in the interval $[0,1]$, in particular $Nz_t /(N+1)<1-\mu$, and the maximum of $D(\mu +x,\mu)$ for $x>0$ is $D(1,\mu)=\log \mu^{-1}$. Hence, equation \ref{eqn:def_z_t} does not have a solution if $f_t$ is too large (that is, if $t$ is small). In these cases we set $z_t =1-\mu$. However, since $f_t$ is decreasing, equation \ref{eqn:def_z_t} does have a solution when $t\geq T$ (for some $T$ depending on $\mu$), and this solution is unique (since $D(\mu +x,\mu )$ is strictly increasing).

With high probability $\hat{\mu}_t -\mu \leq z_t$ for all
$t\in \N$ by Theorem~\ref{thm:anytime_1}. Furthermore, the function $D(\mu +x,\mu )$ is
increasing in $x\geq 0$. By combining these facts we get that with
probability at least $1-\delta$
\[
D \left( \mu +\tfrac{N}{N+1} z_t ,\mu \right) \ \geq \ D \left( \tfrac{N
    \hat{\mu}_t +\mu}{N+1} ,\mu \right) \ .
\]
On the other hand
\[
D \left( \mu +\tfrac{N}{N+1} z_t ,\mu \right) \ \leq \ \frac{\log \left( \kappa (N) \log_2 (2t) /\delta \right)}{t} \ ,
\]
by definition. Chaining these two inequalities leads to the lower
confidence bound 
\begin{equation}\label{lcb}
L(t,\delta ) = \inf \left\{ m<\hat{\mu}_t :\ D \left( \tfrac{N \hat{\mu}_t +m}{N+1} ,m \right) \leq \frac{\log \left( \kappa (N) \log_2 (2t) /\delta \right)}{t} \right\}
\end{equation}
which holds for all times $t$ with probability at least
$1-\delta$. Considering the left deviations of $\hat{\mu}_t -\mu$ we can get an upper confidence bound in a similar manner:
\begin{equation}\label{ucb}
U(t,\delta ) = \sup \left\{ m>\hat{\mu}_t :\ D \left( \tfrac{N \hat{\mu}_t +m}{N+1} ,m \right) \leq \frac{\log \left( \kappa (N) \log_2 (2t) /\delta \right)}{t} \right\} \ .
\end{equation}
That is, for all times $t$, with probability at least $1-2\delta$ we have
$L(t,\delta ) \ \leq \ \hat{\mu}_t \ \leq \ U(t,\delta ).$

Note that the constant $\log \kappa(N) \approx 2 \log_2(N)$, so the
choice of $N$ plays a relatively mild role in the bounds. However, we note here that if $N$ is
sufficiently large, then
$\frac{N \hat{\mu}_t +m}{N+1} \approx \hat{\mu}_t$, and thus
$D \left( \frac{N \hat{\mu}_t +m}{N+1} ,m \right) \approx D \left(
  \hat{\mu}_t,m \right)$, in which case the bounds above are easily
compared to those in prior works
\cite{KLUCB_2011_Garivier,maillard2011finite,cappe2013kullback}. We
make this connection more precise and show that the confidence intervals defined as
\begin{align*}
L'(t,\delta ) & = \inf \left\{ m<\hat{\mu}_t :\ D \left( \hat{\mu}_t,m \right) \leq \frac{c(N) \log \left( \kappa (N) \log_2 (2t) /\delta \right)}{t} \right\}
 \ , \textrm{ and} \\
U'(t,\delta ) & = \inf \left\{ m>\hat{\mu}_t :\ D \left( \hat{\mu}_t,m \right) \leq \frac{c(N) \log \left( \kappa (N) \log_2 (2t) /\delta \right)}{t} \right\}
 \ ,
\end{align*}
satisfy $L'(t,\delta ) \ \leq \ \hat{\mu}_t \ \leq \ U'(t,\delta )$ for all $t$, with probability $1-2\delta$. The constant $c(N)$ is defined in Theorem~\ref{thm:anytime_2} in the Supplementary Material, where the correctness of $L'(t,\delta )$ and $U'(t,\delta )$ is shown.

\begin{proof}[Proof of Theorem~\ref{thm:anytime_1}]
The proofs of parts \emph{(i)} and \emph{(ii)} are completely analogous, hence in what follows we only prove part \emph{(i)}. Note that $\{ \hat{\mu}_t - \mu >z_t \} \iff \{ S_t >tz_t \}$, where $S_t = \sum_{j\in [t]} (Y_j -\mu )$ denotes the centered sum up to time $t$. We start with a simple union bound
\begin{equation}\label{eqn:union_bound}
\P \left( \exists t\in \N :\ S_t > t z_t \right) \leq \P \left( \exists t\in [N]:\ S_t > tz_t \right) + \sum_{k\geq l} \P \left( \exists t\in [2^k ,2^{k+1}]:\ S_t > tz_t \right) \ .
\end{equation}
First, we bound each summand in the second term individually. In an effort to save space, we define the event $A_k= \{ \exists t\in [2^k ,2^{k+1}]:\ S_t > tz_t \}$. Let $t_{j,k} = (1+\frac{j}{N})2^k$. In what follows we use the notation $t_j \equiv t_{j,k}$. We have
\begin{align*}
\P \left( A_k \right) & \leq \sum_{j\in [N]} \P \left( \exists t\in
                        [t_{j-1},t_{j} ]:\ S_t > t z_t \right) \ \leq
                        \ \sum_{j\in [N]} \P \left( \exists t\in [t_{j-1},t_{j}]:\ S_t > t_{j-1}z_{t_{j-1}} \right) \ ,
\end{align*}
where the last step is true if $tz_t$ is non-decreasing in $t$. This technical claim is formally shown in Lemma~\ref{lem:tzt_increasing} in the Supplementary Material. However, to give a short heuristic, it is easy to see that $tz_t$ has an increasing lower bound. Noting that $D(\mu +x,\mu )$ is convex in $x$ (the second derivative is positive), and that $D(\mu ,\mu )=0$, we have $D(1,\mu ) x\geq D(\mu +x,\mu )$. Hence $z_t \gtrsim t^{-1} \log \log t$.

Using a Chernoff-type bound together with Doob's inequality, we can continue as
\begin{align}\label{eqn:kl_to_subgauss_2}
\P \left( A_k \right) & \leq \inf_{\lambda >0} \sum_{j\in [N]} \P \left( \exists t\in [t_{j-1},t_{j}]:\ \exp \left( \lambda S_t \right) > \exp \left( \lambda t_{j-1} z_{t_{j-1}} \right) \right) \nonumber \\
& \leq \sum_{j\in [N]} \exp \left( -\sup_{\lambda >0} \left( \lambda t_{j-1} z_{t_{j-1}} - \log \E \left( e^{\lambda S_{t_{j}}} \right) \right) \right) \nonumber \\
& = \sum_{j\in [N]} \exp \left( -t_{j} \sup_{\lambda \geq 0} \left( \lambda \tfrac{N+j-1}{N+j} z_{t_{j-1}} - \log \E \left( e^{\lambda (Y_1 -\mu )} \right) \right) \right) \ .
\end{align}

Using $\E (e^{\lambda Y_1}) \leq \E (e^{\lambda \xi})$ where $\xi \sim \Ber (\mu )$ (see Lemma~9 of \cite{KLUCB_2011_Garivier}), and the notation $\alpha_j= \tfrac{N+j-1}{N+j}$,
\begin{align}\label{eqn:kl_to_subgauss}
\P \left( A_k \right) & \leq \sum_{j\in [N]} \exp \left( -t_{j} \sup_{\lambda \geq 0} \left( \lambda \alpha_j z_{t_{j-1}} - \log \E \left( e^{\lambda (\xi -\mu )} \right) \right) \right) \nonumber \\
& = \sum_{j\in [N]} \exp \left( -t_{j} D \left( \mu + \alpha_j z_{t_{j-1}},\mu \right) \right) \ ,
\end{align}
since the rate function of a Bernoulli random variable can be explicitly computed, namely we have $\sup_{\lambda >0} (\lambda x - \log \E (e^{\lambda \xi})) = D(\mu +x,\mu )$ (see \cite{Concentration_2013}).

Again, we use the convexity of $D(\mu +x,\mu)$. For any $\alpha \in (0,1)$ we have $\alpha D(\mu +x,\mu ) \geq D(\mu +\alpha x,\mu )$, since $D(\mu ,\mu )=0$. Using this with $\alpha = \tfrac{N}{\alpha_j (N+1)}$ and $x= \alpha_j z_{t_{j-1}}$, we get that
\[
\tfrac{N}{\alpha_j (N+1)} D \left( \mu + \alpha_j z_{t_{j-1}},\mu \right) \geq D \left( \mu +\tfrac{N}{N+1} z_{t_{j-1}},\mu \right) \ .
\]

This implies
\begin{equation}\label{eqn:key_to_thm2}
\P \left( A_k \right) \leq \sum_{j\in [N]} \exp \left( -t_{j} \tfrac{N+1}{N} \alpha_j D \left( \mu + \tfrac{N}{N+1} z_{t_{j-1} },\mu \right) \right) \ .
\end{equation}

Plugging in the definition of $t_j$ and the sequence $z_t$, and noting that $\delta <1$, we arrive at the bound
\[
\P \left( A_k \right) \leq \sum_{j\in [N]} \exp \left( - \frac{N+1}{N} \log \left( \kappa (N ) \log_2 (2^{k+1} \frac{N+j-1}{N}) /\delta \right) \right) \leq N \left( \tfrac{\delta}{\kappa(N) (k+1)} \right)^{\frac{N+1}{N}} \ .
\]

Regarding the first term in (\ref{eqn:union_bound}), again using the Bernoulli rate function bound we have
\[
\P \left( \exists t\in [N]:\ \hat{\mu}_t - \mu > z_t \right) \leq \sum_{t\in [N]} \P \left( \hat{\mu}_t - \mu > z_t \right) \leq \sum_{t\in [N]} \exp \left( -t D(\mu +z_t,\mu ) \right) \ .
\]

Using the convexity of $D(\mu +x,\mu )$ as before, we can continue as
\begin{align*}
\P \left( \exists t\in [N]:\ \hat{\mu}_t - \mu > z_t \right) & \leq \sum_{t\in [N]} \exp \left( -t \tfrac{N+1}{N} D\left( \mu +\tfrac{N}{N+1} z_t,\mu \right) \right) \\
& \leq \sum_{t\in [N]} \exp \left( - \tfrac{N+1}{N} \log \left( \kappa (N) \log_2 (2t) /\delta \right) \right) \\
& \leq \delta^{\frac{N+1}{N}} \kappa (N)^{-\frac{N+1}{N}} \sum_{t\in [N]} \log_2 (2t)^{-\frac{N+1}{N}} \ .
\end{align*}

Plugging the two bounds back into (\ref{eqn:union_bound}) we conclude that
\[
\P \left( \exists t:\ \hat{\mu}_t - \mu > z_t \right) \leq \delta^{\frac{N+1}{N}} \kappa (N)^{-\frac{N+1}{N}} \sum_{j\in [N]} \left( \1_{\{ l\neq 0 \}} \log_2 (2j)^{-\frac{N+1}{N}} + \sum_{k\geq l} (k+1)^{-\frac{N+1}{N}} \right) \leq \delta \ ,
\]
by the definition of $\kappa (N)$.
\end{proof}


\section{Analysis of lil-KLUCB}\label{sec:lil-KLUCB}

Recall that the lil-KLUCB algorithm uses confidence bounds of the form $U_i (t,\delta) = \sup \{ m>\hat{\mu}_t : D(\hat{\mu}_t ,m) \leq f_t(\delta ) \}$ with some decreasing sequence $f_t(\delta )$. In this section we make this dependence explicit, and use the notations $U_i (f_t(\delta ))$ and $L_i (f_t(\delta ))$ for upper and lower confidence bounds.
For any $\epsilon >0$ and $i\in [n]$, define the events $\Omega_i (\epsilon ) = \{ \forall t\in \N :\ \mu_i \in [L_i (f_t(\epsilon )),U_i (f_t(\epsilon ))] \}$.

The correctness of the algorithm follows from the correctness of the
individual confidence intervals, as is usually the case with LUCB
algorithms. This is shown formally in
Proposition~\ref{prop:correctness} provided in the Supplementary
Materials.  The main focus in this section is to show a high
probability upper bound on the sample complexity. This can be done by
combining arguments frequently used for analyzing LUCB algorithms and
those used in the analysis of the lil-UCB \cite{lilUCB_2014}. The
proof is very similar in spirit to that of the LUCB++ algorithm
\cite{Kevin_2017}. Due to spatial restrictions, we only provide a proof
sketch here, while the detailed proof is provided in the Supplementary
Materials.

\begin{proof}[Proof sketch of Theorem~\ref{thm:samplecomp}]
  Observe that at each time step two things can happen (apart from
  stopping): \emph{(1)} Arm~1 is not sampled (two sub-optimal arms are
  sampled); \emph{(2)} Arm~1 is sampled together with some other
  (suboptimal) arm.  Our aim is to upper bound the number of times any
  given arm is sampled for either of the reasons above. We do so
  by conditioning on the event
\[
\Omega' = \Omega_1 (\delta ) \cap \left( \bigcap_{i\geq 2} \Omega_i
  (\delta_i ) \right) \ , \ \mbox{for a certain choice of $\{\delta_i\}$
  defined below.} 
\]
For instance, if arm~$1$ is not sampled at a given time $t$, we know
that $\TOP (t) \neq 1$, which means there must be an arm $i\geq 2$ such
that $U_i (T_i(t),\delta )\geq U_1(T_1(t),\delta )$. However, on the
event $\Omega_1 (\delta )$, the UCB of arm~1 is accurate, implying
that $U_i(T_i(t),\delta )\geq \mu_1$. This implies that $T_i(t)$ can
not be too big, since on $\Omega_i (\delta_i )$, $\hat{\mu}_{i,t}$ is
``close" to $\mu_i$, and also $U_i(T_i(t),\delta )$ is not much larger
then $\hat{\mu}_i$. All this is made formal in Lemma~\ref{lem:stop},
yielding the following upper bound on number of times arm~$i$ is sampled for
reason \emph{(1)}:
\[
\tau_i (\delta \cdot \delta_i ) = \min \left\{ t\in \N :\ f_t (\delta \cdot \delta_i ) < D^* (\mu_i ,\mu_1 ) \right\} \ .
\]

Similar arguments can be made about the number of samples of any
suboptimal arm~$i$ for reason~\emph{(2)}, and also the number of samples
on arm~$1$. This results in the sample complexity upper bound
\[
\frac{K_1 \log \left( (n-1) \delta^{-1} \log D^* (\mu_1 ,\tilde{\mu})^{-1} \right)}{D^* (\mu_1 ,\tilde{\mu})} + \sum_{i\geq 2} \frac{K_1 \log \left( \delta^{-1} \log D^* (\mu_i ,\tilde{\mu}_i )^{-1} \right) + \log \delta_i^{-1}}{D^* (\mu_i ,\tilde{\mu}_i )} \ ,
\]
on the event $\Omega'$, where $K_1$ is a universal constant.  Finally,
we define the quantities
$\delta_i = \sup \{ \epsilon >0: U_i (f_t(\epsilon )) \geq \mu_i \
\forall t\in \N \}$. Note that we have $\P (\delta_i < \gamma ) = \P (\exists t\in \N :\ U_i (f_t(\gamma )) \geq \mu_i ) \leq \gamma$ according to Theorem~\ref{thm:anytime_2} in the Supplementary Material. Substituting $\gamma =\exp (- D^* (\mu_i ,\tilde{\mu}_i ) z)$ we get
\[
\P \left( \tfrac{\log \delta_i^{-1}}{D^* (\mu_i ,\tilde{\mu}_i )} \geq z \right) \leq \exp (-D^* (\mu_i ,\tilde{\mu}_i ) z) \ .
\]
Hence $\{ \delta_i \}_{i\geq 2}$ are independent sub-exponential variables, which allows us to control their contribution to the sum above using standard techniques.
\end{proof}


\section{Real-World Crowdsourcing}\label{sec:experiment}


We now compare the performance of {\bf lil-KLUCB} to that of other
algorithms in the literature. We do this using both synthetic data and
real data from the New Yorker Cartoon Caption contest \cite{cnet}\footnote{These data can be found at \url{https://github.com/nextml/caption-contest-data}}. To keep comparisons fair, we run the same UCB algorithm for
all the competing confidence bounds.
We set $N = 8$ and $\delta = 0.01$ in our experiments. The confidence bounds
are {\bf [KL]}: the KL-bound
derived based on Theorem~\ref{thm:anytime_1},
{\bf [SG1]}: a matching sub-Gaussian bound
derived using the proof of Theorem~\ref{thm:anytime_1}, using
sub-Gaussian tails instead of the KL rate-function (the exact
derivations are in the Supplementary Material), and
{\bf [SG2]}: the sharper sub-Gaussian bound provided by Theorem~8 of
\cite{kaufmann2016complexity}.


We compare these methods by computing the empirical probability that
the best-arm is among the top 5 empirically best arms, as a function
of the total number of samples. We do so using using synthetic data in
Figure~\ref{fig:parametric} , where the Bernoulli rewards simulate
cases from Table~\ref{tab:KL_vs_SG}, and using real human response data from two
representative New Yorker caption contests in Figure~\ref{fig:NYr}.

\begin{figure}[h]\label{fig:parametric}
\begin{center}
\begin{adjustbox}{width=.85\textwidth}
\centering
\includegraphics[width = 2in]{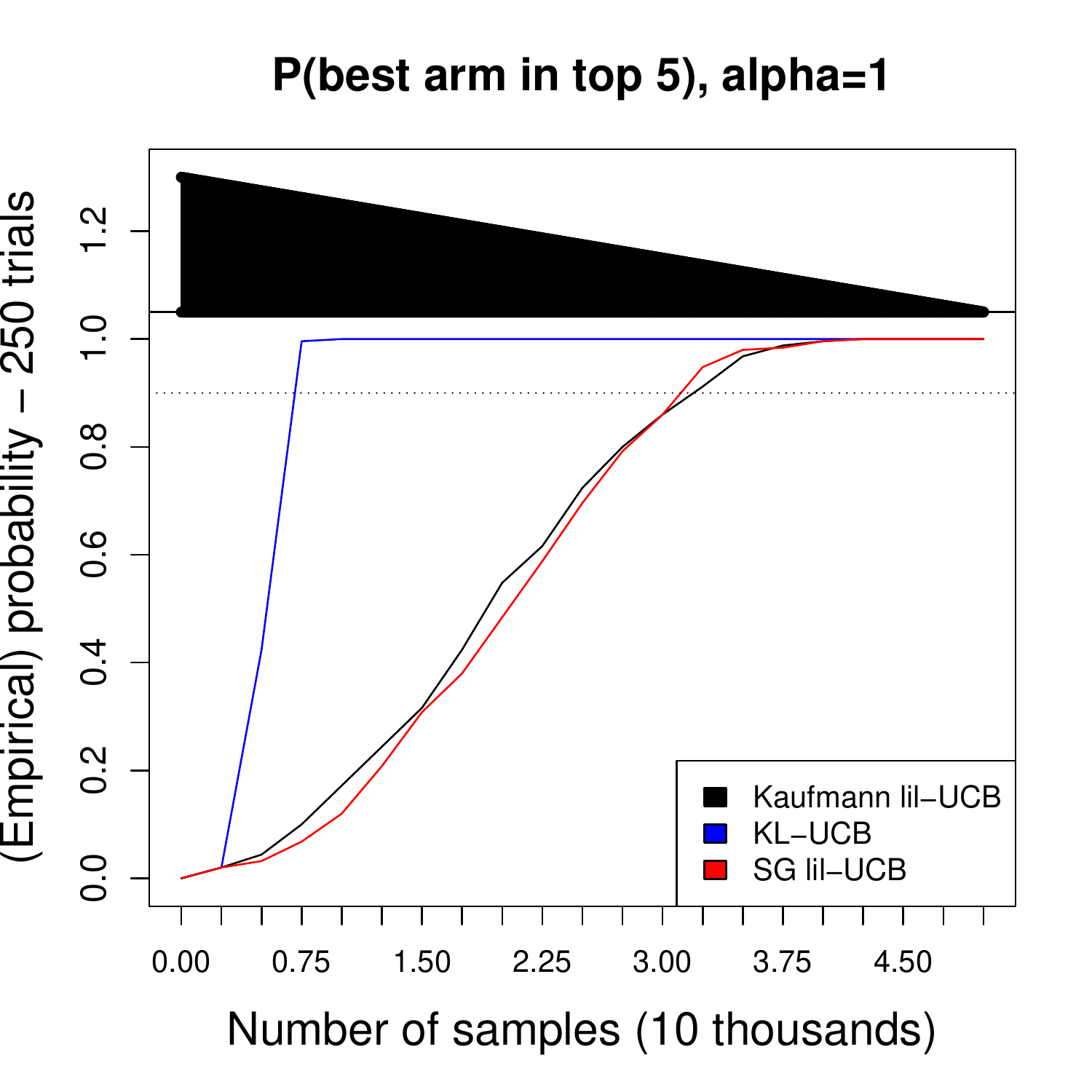} \quad
\includegraphics[width = 2in]{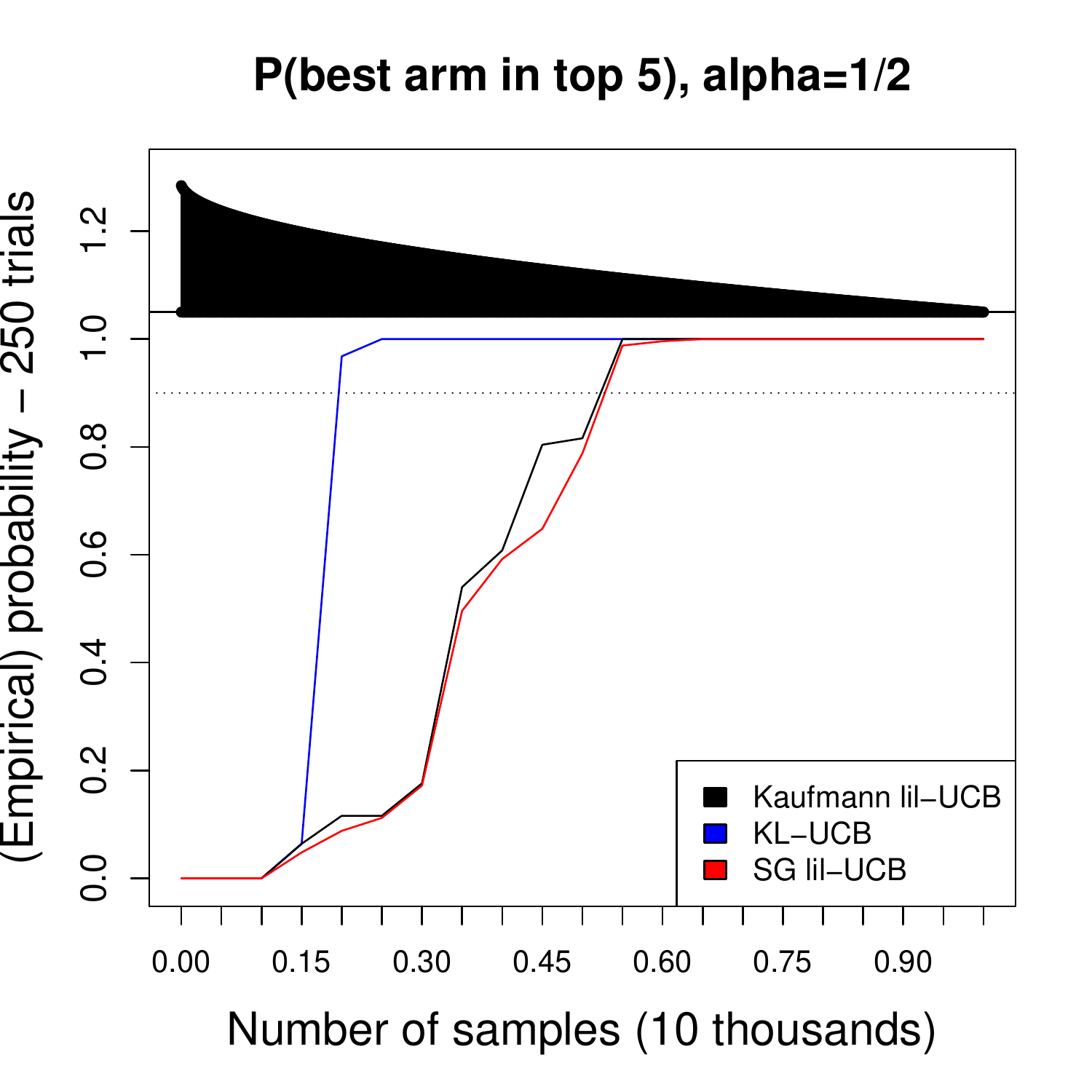}
\end{adjustbox}
\caption{{\small Probability of the best-arm in the top 5 empirically
  best arms, as a function of the number of samples, based on $250$
  repetitions. $\mu_i = 1-((i-1)/n)^\alpha$, with $\alpha=1$ in the
  left panel, and $\alpha=1/2$ in the right panel. The mean-profile is
  shown above each plot. {\bf [KL]} Blue;  {\bf [SG1]} Red;  {\bf [SG2]} Black.}}
\end{center}
\end{figure}

As seen in Table~\ref{tab:KL_vs_SG}, the KL confidence bounds have the
potential to greatly outperform the sub-Gaussian ones. To illustrate
this indeed translates into superior performance, we simulate two
cases, with means $\mu_i =1-((i-1)/n)^\alpha$, with $\alpha =1/2$ and
$\alpha=1$, and $n=1000$. As expected, the KL-based method
requires significantly fewer samples (about $20\ \%$ for
$\alpha=1$ and $30\ \%$ for $\alpha=1/2$) to find the best arm. Furthermore, the
arms with means below the median are sampled about $15$ and $25\ \%$ of the time
respectively -- key in crowdsourcing applications, since having participants answer
fewer irrelevant (and potentially annoying) questions improves both
efficiency and user experience.

\begin{figure}[h]\label{fig:NYr}
\begin{center}
\begin{adjustbox}{width=.85\textwidth}
\centering
\includegraphics[width = 2in]{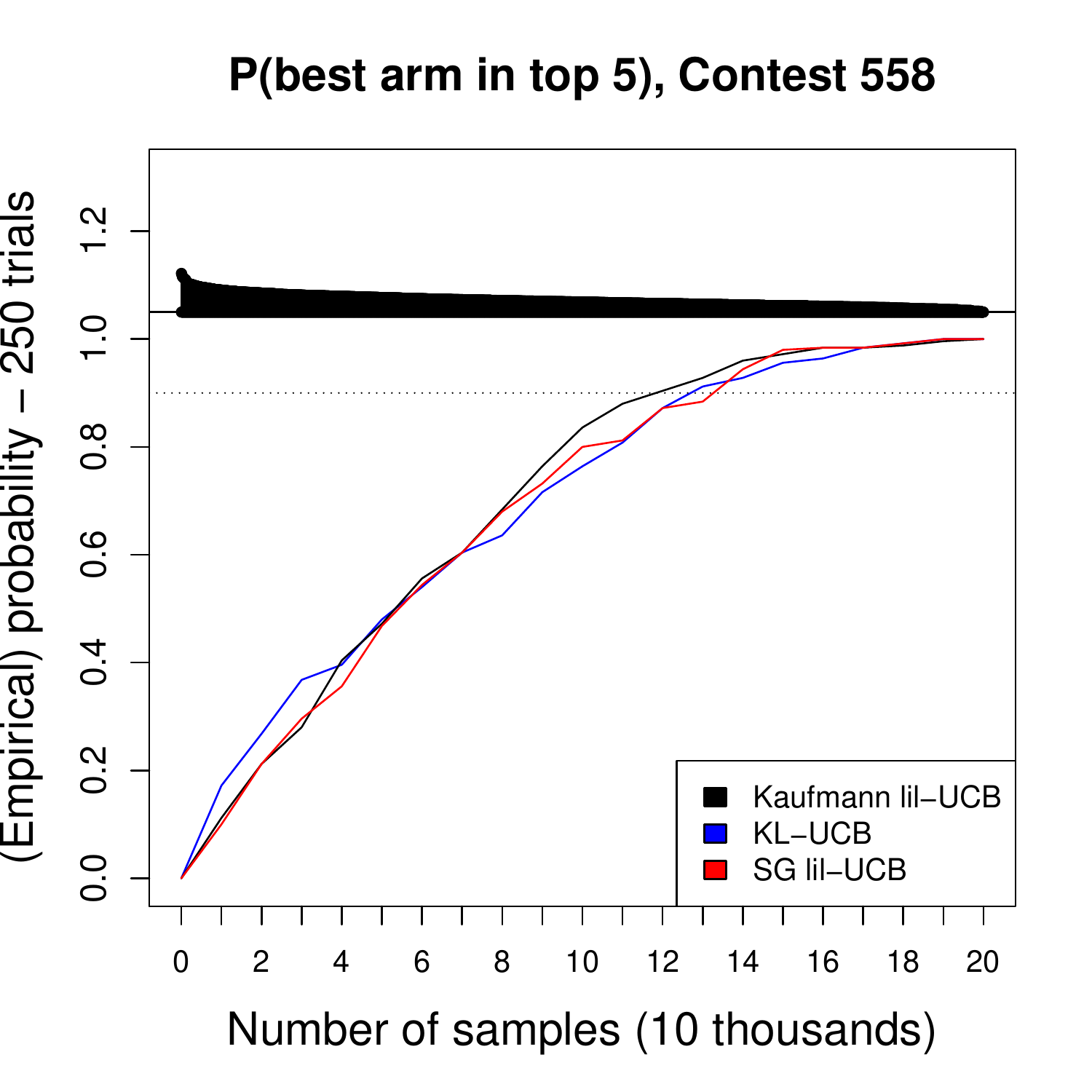} \quad
\includegraphics[width = 2in]{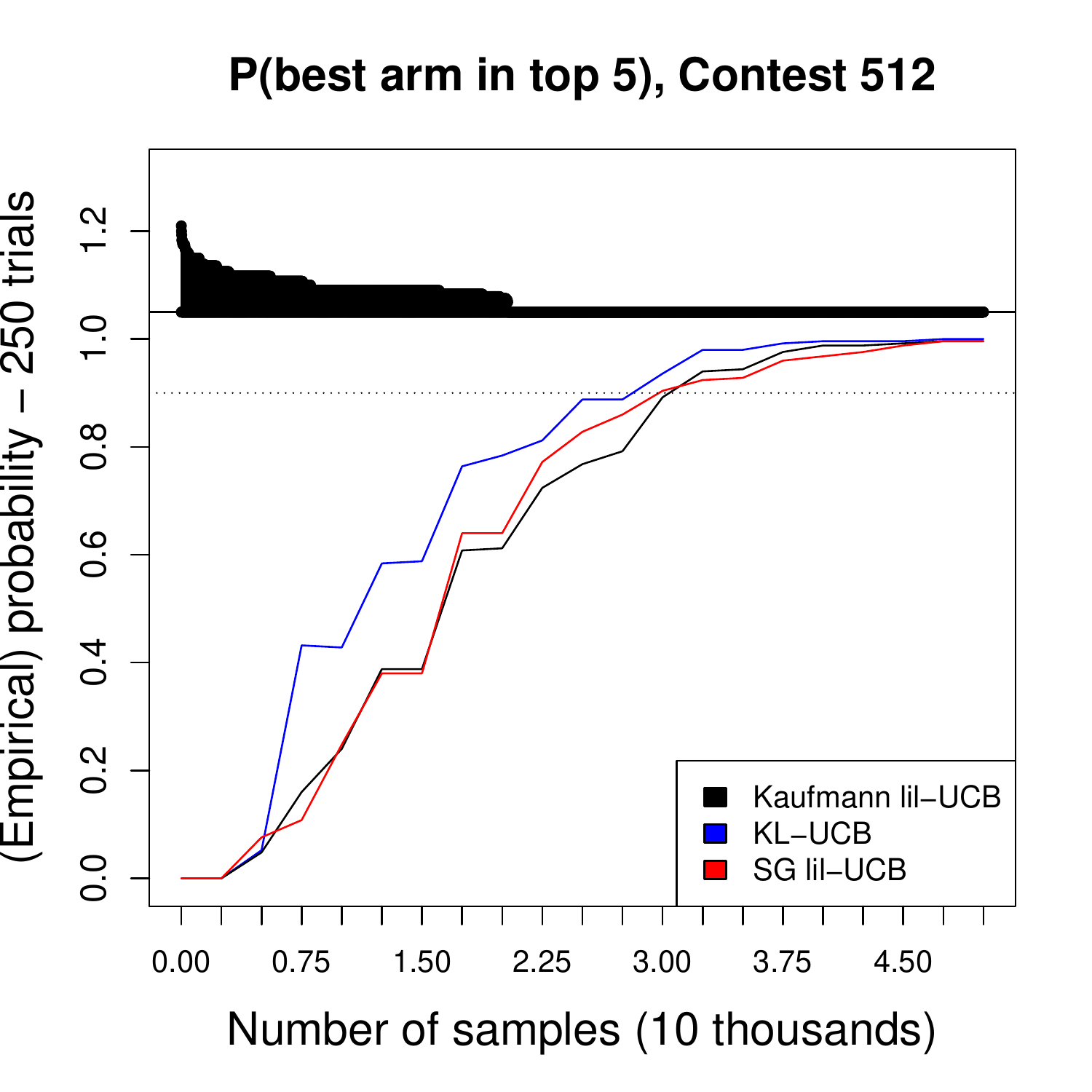}
\end{adjustbox}
\caption{{\small Probability of the best-arm in the top 5 empirically
  best arms vs.\ number of samples, based on $250$
  bootstrapped repetitions. Data from New Yorker contest 558 ($\mu_1 =
  0.536$) on left, and contest 512 ($\mu_1 = 0.8$) on
  right. Mean-profile above each plot.  {\bf [KL]} Blue;  {\bf [SG1]} Red;  {\bf [SG2]} Black. }}
\end{center}
\end{figure}

To see how these methods fair on real data, we also run these
algorithms on bootstrapped human response data from the real New
Yorker Caption Contest. The mean reward of the best arm in these
contests is usually between $0.5$ and $0.85$, hence we choose one
contest from each end of this spectrum. At the lower end of the
spectrum, the three methods fair
comparably.  This is expected because the sub-Gaussian bounds are
relatively good for means about $0.5$.  However, in cases where the
top mean is significantly larger than $0.5$ we see a marked
improvement in the KL-based algorithm.

\bibliographystyle{plain}
\bibliography{KLUCB_references}


\newpage
\appendix

\section{Proofs for Section~\ref{sec:anytime}}\label{app:anytime}

\begin{lemma}\label{lem:tzt_increasing}
Let $T$ be the first time index such that
(\ref{eqn:def_z_t}) has a solution. Since $z_t =1-\mu$ by definition for $t<T$, clearly $tz_t$ is increasing for $t\in [T-1]$. 

Now consider the case $t\geq T-1$. Using the convexity of $D(\mu +x,\mu )$ (in $x$) and the definition of the sequence $z_t$, we have
\begin{align*}
D \left( \mu+ \frac{t}{t+1} \frac{N}{N+1} z_t ,\mu \right) & \leq \frac{t}{t+1} D \left( \mu+ \frac{N}{N+1} z_t ,\mu \right) \\
& \leq \frac{t}{t+1} \frac{\log \left( \kappa (N) \log_2 (2t) /\delta \right)}{t} \\
& \leq \frac{\log \left( \kappa (N) \log_2 (2(t+1)) /\delta \right)}{t+1} \\
& = D \left( \mu+\frac{N}{N+1} z_{t+1},\mu \right) \ ,
\end{align*}
where the last equality holds, since $t\geq T-1$. Comparing the two ends of this chain of inequalities implies that $\frac{t}{t+1} z_t \leq z_{t+1}$ since the function $D(\mu +x,\mu )$ is increasing in $x$.
\end{lemma}

\begin{theorem}\label{thm:anytime_2}
Consider the setting of Theorem~\ref{thm:anytime_1} and let
\[
c(N) = \frac{N+1}{N-\log (N+1)} \ .
\]
Define $z_t$ as the solution of
\[
D(\mu + z_t,\mu ) = \frac{c(N) \log \left( \kappa (N) \log_2 (2t)/\delta \right)}{t} \ ,
\]
if a solution exists, and $z_t =1-\mu$ otherwise. Then
\begin{align*}
(i):\ & \P \left( \exists t\in \N :\ \hat{\mu}_t -\mu > z_t \right) \leq \delta \ , \\
(ii):\ & \P \left( \exists t\in \N :\ \hat{\mu}_t -\mu < -z_t \right) \leq \delta \ .
\end{align*}
\end{theorem}

The correctness of the confidence intervals $L'(t,\delta )$ and $U'(t,\delta )$ follow from Theorem~\ref{thm:anytime_2} in the same way as that of $L(t,\delta )$ and $U(t,\delta )$ follow from Theorem~\ref{thm:anytime_1} shown in Section~\ref{sec:anytime}.

\begin{proof}[Proof of Theorem~\ref{thm:anytime_2}]
It is clear by consulting the proof of Theorem~\ref{thm:anytime_1} that if we had
\[
D(\mu +x,\mu ) \leq c(N) D \left( \mu + \frac{N}{N+1} x,\mu \right) \quad \forall x\in [0,1-\mu ],\ \forall \mu \in (0,1) \ ,
\]
then using it at step \eqref{eqn:key_to_thm2} would yield the desired result.

Let $\alpha \in (0,1)$ and use the notation $D(\mu +x,\mu ) = f_\mu (x)$. We wish to show that
\[
g_\mu (x) := f_\mu (x) -c f_\mu (\alpha x) \leq 0 \quad \forall x\in [0,1-\mu ],\ \forall \mu \in (0,1) \ .
\]
with $c=\frac{1}{\alpha + (1-\alpha ) \log (1-\alpha)}$.

We first examine $g_\mu (x)$ as a function of $x$. Recall that the first and second derivatives of $f_\mu (x)$ (in $x$) are
\[
f'_\mu (x) = \log \frac{\mu +x}{\mu} - \log \frac{1-\mu -x}{1-\mu} \ ,\ \textrm{and}\ f''_\mu (x) = \frac{1}{(\mu +x)(1-\mu -x)} \ .
\]

Hence
\[
g''_\mu (x) = f''_\mu (x) - c\alpha^2 f''_\mu (\alpha x) \ .
\]
As for the sign of the second derivative, we have
\begin{align*}
f''_\mu (x) - c\alpha^2 f''_\mu (\alpha x) & \lessgtr 0 \\
& \Updownarrow \\
f''_\mu (x) & \lessgtr c \alpha^2 f''_\mu (\alpha x) \\
& \Updownarrow \\
(\mu +\alpha x)(1-\mu -\alpha x) & \lessgtr c\alpha^2 (\mu +x)(1-\mu -x) \\
& \Updownarrow \\
(c-1)\alpha^2 x^2 + \alpha (1-2\mu )(1-c\alpha )x + \mu (1-\mu )(1-c\alpha^2 ) & \lessgtr 0 \ .
\end{align*}

Denote the left side by $h(x)$. The roots of $h(x)$ are
\begin{align*}
x_{1,2} & = \frac{-\alpha (1-2\mu )(1-c\alpha ) \pm \sqrt{\alpha^2 (1-2\mu )^2(1-c\alpha )^2 - 4 \alpha (1-2\mu )(1-c\alpha )\mu (1-\mu )(1-c\alpha^2 )}}{2 \alpha (1-2\mu )(1-c\alpha )} \\
& = \frac{-(1-2\mu )(1-c\alpha ) \pm \sqrt{(1-2\mu )^2(1-c\alpha )^2 - 4 (c-1)\mu (1-\mu )(1-c\alpha^2 )}}{2 (c-1)\alpha} \ .
\end{align*}
Note that
\[
c = \frac{1}{\alpha + (1-\alpha ) \log (1-\alpha)} \geq \frac{1}{\alpha - (1-\alpha )\alpha} = \frac{1}{\alpha^2} > \frac{1}{\alpha} >1 \ ,
\]
since $\log (1+x)\leq x$. This implies that the expression under the root is positive, and that
\[
\sqrt{(1-2\mu )^2(1-c\alpha )^2 - 4 (c-1)\mu (1-\mu )(1-c\alpha^2 )} \geq |(1-2\mu )(1-c\alpha )| \ ,
\]
which in turn implies that at least one of the roots of $h(x)$ is negative. Let $y=\max \{ x_1,x_2 \}$.

By the previous observation, the function $g_\mu (x)$ is concave on the interval $[0,y]$ and convex on the interval $[y,1-\mu ]$ (with the convention that $[a,b]=\emptyset$ if $a>b$). Noting that $g_\mu (0)=0$ and $g'_\mu (0)=0$ we have $g_\mu (x)\leq 0$ on $[0,y]$. On the other hand, $g_\mu (1-\mu )\leq 0 \Rightarrow g_\mu (x) \leq 0$ for $x\in [y,1-\mu ]$, by the convexity of $g_\mu (x)$ on this interval and that $g_\mu (y)\leq 0$.

Hence, all that remains to show is $g_\mu (1-\mu )\leq 0$ for all $\mu \in (0,1)$. This yields the inequality
\begin{align*}
0 & \geq \log \frac{1}{\mu} - c \left( \left( \mu + \alpha (1-\mu ) \right) \log \frac{\mu + \alpha (1-\mu )}{\mu } + \left( 1-\mu -\alpha (1-\mu ) \right) \log \frac{1-\mu -\alpha (1-\mu )}{1-\alpha} \right) \\
& = \log \frac{1}{\mu} - c \left( \left( (1-\alpha ) \mu + \alpha \right) \log \left( 1+ \frac{\alpha}{(1-\alpha )\mu} \right) + \log (1-\alpha ) \right):= l(\mu )\ .
\end{align*}

Note that the right side is equal to zero at $\mu =1$. To conclude the inequality above, we show that the right side is increasing in $\mu$. We have
\begin{align*}
\frac{\partial}{\partial \mu} l(\mu ) & = -\frac{1}{\mu} -c \left( (1-\alpha ) \log \left( 1+\frac{\alpha}{(1-\alpha )\mu} \right) -\frac{\alpha}{\mu} \right) \\
& = \frac{1}{\mu} \left( c\alpha -1 \right) -c(1-\alpha ) \log \left( 1+\frac{\alpha}{(1-\alpha )\mu} \right) \ .
\end{align*}

Using the inequality $\log (1+x) \leq \frac{x-a}{a+1} + \log (1+a)$ (that is the line tangential to $\log (1+x)$ at any point $a>-1$ upper bounds $\log (1+x)$) with $a=\frac{\alpha}{1-\alpha}$, we can continue as
\begin{align*}
\frac{\partial}{\partial \mu} l(\mu ) & \geq \frac{1}{\mu} \left( c\alpha -1 \right) -c(1-\alpha ) \log \left( \alpha \left( \frac{1}{\mu} -1 \right) - \log (1-\alpha ) \right) \\
& = \frac{1}{\mu} \left( c\alpha -1-c\alpha (1-\alpha ) \right) - c(1-\alpha ) \left( \alpha - \log (1-\alpha ) \right) \ .
\end{align*}

Finally, noting that $c\alpha -1-c\alpha (1-\alpha )$ is positive (since $c\geq 1/\alpha^2$) we can further decrease the right hand side by using $1/\mu \geq 1$, which yields
\[
\frac{\partial}{\partial \mu} l(\mu ) \geq c \left( \alpha -(1-\alpha )\log (1-\alpha ) \right) -1 \ .
\]
The right side is non-negative, whenever $c\geq 1/(\alpha - (1-\alpha )\log (1-\alpha ))$, concluding the proof.
\end{proof}


\section{Proofs for Section~\ref{sec:lil-KLUCB}}\label{app:lil-KLUCB}

\begin{proposition}\label{prop:correctness}
The lil-KLUCB algorithm is $2 \delta$-PAC. 
\end{proposition}
\begin{proof}
Suppose that when the algorithm stops, $\TOP (t) \neq 1$. This implies that there exists $t\in \N$ and $i\geq 2$ for which
\[
L_i (f_{T_i(t)}(\delta /(n-1))) > U_1 (f_{T_1(t)}(\delta )) \ .
\]
Consider the events $\Omega_1 (\delta)$ and $\Omega_i (\delta /(n-1))$ for $i \geq 2$, and let their intersection be
\[
\Omega = \Omega_1 (\delta ) \cap (\cap_{i\geq 2} \Omega_i (\delta /(n-1))) \ .
\]
Note that
\[
\P (\Omega ) = 1- \P (\overline{\Omega}) \geq 1-\P \big( \overline{\Omega}_1 (\delta ) \cup ( \cup_{i\geq 2} \overline{\Omega}_i (\delta /(n-1)) \big) \geq 1-2\delta
\]
by Theorem~\ref{thm:anytime_2} (where $\overline{\Omega}$ is the complementary event of $\Omega$). However, on the event $\Omega$ the algorithm cannot fail, as on this event $L_i (f_{T_i(t)}(\delta /(n-1))) \leq \mu_i$ and $U_1 (f_{T_1(t)}(\delta )) \geq \mu_1$ which (together with the first display) would imply $\mu_1 < \mu_i$, a contradiction.
\end{proof}

The backbone to proving Theorem~\ref{thm:samplecomp} is the following lemma. Recall that for $\mu ,\tilde{\mu}\in [0,1]$, the Chernoff information $D^* (\mu ,\tilde{\mu})$ between two Bernoulli random variables with parameters $\mu$ and $\tilde{\mu}$ can be written as
\[
D^* (\mu ,\tilde{\mu}) = \inf_{x\in (0,1)} \max \left\{ D(x,\mu ),D(x,\tilde{\mu}) \right\} \ .
\]

\begin{lemma}\label{lem:stop}
Let $Y_1,Y_2,\dots$ be independent samples from a distribution $\P$, and consider a sequence of confidence bounds for the mean $\mu$ of the form
\[
U (f_t(\delta )) = \sup \left\{ m>\hat{\mu}_t :\ D(\hat{\mu}_t ,m) \leq f_t(\delta ) \right\} \ ,
\]
where $\hat{\mu}_t$ is the empirical mean based on $\{ Y_j \}_{j\in [t]}$, $\delta \in (0,1)$ and $f_t (x)$ is decreasing in $x$. Consider a realization of the sequence $\{ \hat{\mu}_t \}_{t\in \N}$, and suppose that $\epsilon \in (0,1)$ is such that
\[
D (\hat{\mu}_t,\mu ) \leq f_t (\epsilon ) \ \forall t\in \N \ .
\]
Then for any fixed $\tilde{\mu} \in (\mu ,1)$ we have
\[
f_t (\delta \cdot \epsilon ) < D^* (\tilde{\mu},\mu )\ \Rightarrow \ U (f_t(\delta )) < \tilde{\mu} \ .
\]
\end{lemma}
\begin{proof}
We first note that $f_t (\delta \cdot \epsilon ) \geq \min \{ f_t (\delta ),f_t (\epsilon ) \}$ since $\delta ,\epsilon \leq 1$ and $f_t (\cdot )$ is decreasing.

The claim then follows by the definitions of $D^* (\mu ,\tilde{\mu} ), U_t (\delta )$ and $\epsilon$. In particular, on one hand $D (\hat{\mu}_t,\mu ) \leq f_t (\delta \cdot \epsilon )$ for every $t\in \N$. On the other hand, 
\[
\tilde{\mu} \leq U_t (\delta )\ \iff \  D(\hat{\mu}_t ,\tilde{\mu}) \leq f_t(\delta )\ \Rightarrow \ D(\hat{\mu}_t ,\tilde{\mu}) \leq f_t(\delta \cdot \epsilon ) \ .
\]
This would imply that for $\hat{\mu}_t$ we both have both $D (\hat{\mu}_t,\mu ) \leq f_t (\delta \cdot \epsilon )$ and $D(\hat{\mu}_t ,\tilde{\mu}) \leq f_t(\delta \cdot \epsilon )$. However, this is impossible, by the definition of $D^* (\tilde{\mu},\mu )$.
\end{proof}

With this lemma, we are ready to prove Theorem~\ref{thm:samplecomp}.

\begin{proof}[Proof of Theorem~\ref{thm:samplecomp}]
Observe that at each time step two things can happen in the algorithm (apart from stopping): \emph{(1)} Arm~1 is not pulled (two sub-optimal arms are pulled); \emph{(2)} Arm~1 is pulled together with some other (suboptimal) arm.

Our aim is to upper bound the number of times any given arm is be played for either of the reasons above. We do so on an event of the form
\[
\Omega' = \bigcap_{i\in [n]} \Omega_i (\delta_i ) \ ,
\]
as a function of the quantities $\{ \delta_i \}_{i\in [n]}$, invoking Lemma~\ref{lem:stop}. We set $\delta_1 =\delta$ and choose $\{ \delta_i \}_{i\geq 2}$ such that they take the largest possible values, i.e. $\delta_i = \sup \{ \epsilon \in (0,1):\  \Omega_i (\epsilon )\ \textrm{holds}\}$. Finally, we control the contribution of these random $\delta_i$ to the sample complexity bound obtained in the previous step.

Note that we know from Theorem~\ref{thm:anytime_2} that $\P ( \overline{\Omega}_1 (\delta ))\leq \delta$.

\vspace{0.2cm}
\textbf{A sample complexity bound under $\Omega'$:} If Arm~1 is not pulled at time $t$, there has to exist another Arm~$i$ such that $\hat{\mu}_{i,t} \geq \hat{\mu}_{1,t}$. Under the event $\Omega_1 (\delta )$ this can no longer happen once $U_i (f_{T_i(t)}(\delta )) < \mu_1$. By Lemma~\ref{lem:stop} the latter is guaranteed when
\[
f_{T_i(t)} (\delta \cdot \delta_i ) < D^* (\mu_i ,\mu_1 ) \ .
\]
Using the notation
\[
\tau_i (\delta \cdot \delta_i ) = \min \left\{ t\in \N :\ f_t (\delta \cdot \delta_i ) < D^* (\mu_i ,\mu_1 ) \right\} \ ,
\]
we know that any suboptimal Arm~$i$ can only be pulled at most $\tau_i (\delta \cdot \delta_i )$ times in a way that it is not pulled together with Arm~1. Hence, Arm~1 will be played eventually.

Suppose that at time $t$ a suboptimal Arm~$i$ ($i\geq 2$) is pulled together with Arm~1. This can only happen if the confidence regions of the means of the two arms overlap at time $t$, i.e. $L_1 (f_{T_1(t)}(\delta )) \leq U_i (f_{T_i(t)}(\delta ))$. However, this is impossible once there exists a value $\tilde{\mu}_i \in (\mu_i ,\mu_1 )$ that separates the two confidence bounds, i.e $U_i (f_{T_i(t)}(\delta )) < \tilde{\mu}_i < L_1 (f_{T_1(t)} (\delta ))$.

According to Lemma~\ref{lem:stop}, this happens once $T_i (t)$ is such that
\[
f_{T_i(t)} (\delta \cdot \delta_i ) \leq D^* (\mu_i ,\tilde{\mu}_i) \ ,
\]
and $T_1(t)$ is such that
\[
f_{T_1(t)} (\delta ) \leq D^* (\mu_1 ,\tilde{\mu}_i ) \ .
\]
Note that in the second inequality, the quantity on the left hand side can indeed be chosen as $f_{T_1(t)} (\delta )$ instead of $f_{T_1(t)} (\delta^2 )$), which can be easily seen by consulting the proof of Lemma~\ref{lem:stop}.

For $i\geq 2$ let
\[
\xi_i (\delta \cdot \delta_i ) = \min \left\{ t\in \N :\ f_t (\delta \cdot \delta_i ) < D^* (\mu_i ,\tilde{\mu}_i ) \right\} \ ,
\]
and
\[
\xi_1 (\delta ) = \min \left\{ t\in \N :\ f_t (\delta /(n-1)) < \min_{i\geq 2} D^* (\mu_1 ,\tilde{\mu}_i ) \right\} \ .
\]
By monotonicity of the Chernoff-information $\xi_i (\delta \cdot \delta_i )\geq \tau_i (\delta \cdot \delta_i )$ for every $i\geq 2$. Thus, Arm~$i$ can not be pulled more than $\xi_i (\delta \cdot \delta_i )$ times.

Hence the sample complexity on the event $\Omega'$ is upper bounded by
\[
\xi_1 (\delta ) + \sum_{i\geq 2} \xi_i (\delta \cdot \delta_i ) \ .
\]

\vspace{0.2cm}
\textbf{Controlling the contribution of the $\delta_i$:} It is easy to check that there exists a universal constant $K_1$ such that
\[
\xi_i (\delta \cdot \delta_i ) \leq \frac{K_1 \log \left( (\delta \cdot \delta_i )^{-1} \log D^* (\mu_i ,\tilde{\mu}_i )^{-1} \right)}{D^* (\mu_i ,\tilde{\mu}_i )} \ .
\]
and
\[
\xi_1 (\delta ) \leq \frac{K_1 \log \left( (n-1) \delta^{-1} \log D^* (\mu_1 ,\tilde{\mu})^{-1} \right)}{D^* (\mu_1 ,\tilde{\mu})} \ .
\]

Now let $\delta_i = \sup \{ \epsilon >0: U_i (f_t(\epsilon )) \geq \mu_i \ \forall t\in \N \}$. We have
\[
\P (\delta_i < \gamma ) = \P (\exists t\in \N :\ U_i (f_t(\epsilon )) \geq \mu_i ) \leq \gamma
\]
according to Theorem~\ref{thm:anytime_2}. Hence, substituting $\gamma =\exp (- D^* (\mu_i ,\tilde{\mu}_i ) z)$ we get
\[
\P \left( \frac{\log \delta_i^{-1}}{D^* (\mu_i ,\tilde{\mu}_i )} \geq z \right) \leq \exp (-D^* (\mu_i ,\tilde{\mu}_i ) z) \ .
\]
Hence $D^* (\mu_i ,\tilde{\mu}_i )^{-1} \log \delta_i^{-1}$ are independent sub-exponential random variables. Using standard techniques for bounding sums of sub-exponential random variables, we have
\[
\P \left( \sum_{i\geq 2} \frac{\log \delta_i^{-1}}{D^* (\mu_i ,\tilde{\mu}_i )} \geq K_2 \sum_{i\geq 2} \frac{\log \delta^{-1}}{D^* (\mu_i ,\tilde{\mu}_i )} \right) \leq \delta \ ,
\]
with some constant $K_2$.

Combining this inequality with those for $\xi_i (\cdot )$ concludes the proof.
\end{proof}


\section{The sub-Gaussian tail-bounds for the numerical comparisons of Section~\ref{sec:experiment}}\label{app:experiment}

We can get a sub-Gaussian tail bound as well with the method of Theorem~\ref{thm:anytime_1} as follows. We start by the same union-bound \ref{eqn:union_bound}.

Upper bounding the terms in the second sum go analogously up to the display \ref{eqn:kl_to_subgauss}. At that point, we can use Pinsker's inequality stating that $2(x-y)^2 \leq D(x,y)$ (see \cite{Tsybakov_2009})\footnote{Note that another approach would be to use Hoeffding's bound for the moment generating function $E(e^{\lambda (Y_1 -\mu )})$ at \ref{eqn:kl_to_subgauss_2}. At the end, this would result in the same result as using Pinsker's inequality.}. This yields
\[
\P \left( \exists t\in [2^k ,2^{k+1}]:\ \hat{\mu}_t -\mu > z_t \right) \leq \exp \left( -2t_j \left( \frac{N+j-1}{N+j} \right)^2 z_{t_{j-1}}^2 \right) \ .
\]
Recall that $t_j = (1+\frac{j}{N})2^k$ and define
\[
z_t = \sqrt{\frac{1}{2} \left( \frac{N+1}{N} \right)^2 \frac{\log \left( \kappa(N) \log_2 (2t)/\delta \right)}{t}} \ ,
\]
where $\kappa(N)$ is the same constant as in the statement of Theorem~\ref{thm:anytime_1}. Note that the sequence $tz_t$ is increasing, which was required for the computations leading to \ref{eqn:kl_to_subgauss}.

Plugging in these values, we get
\begin{align*}
\P & \left( \exists t\in [2^k ,2^{k+1}]:\ \hat{\mu}_t -\mu > z_t \right) \\
& \leq \exp \left( -\frac{N+j-1}{N+1} \left( \frac{N+1}{N} \right)^2 \log \left( \kappa(N) \log_2 \left( 2^{k+1} \frac{N+j}{N} \right) /\delta \right) \right) \\
& \leq \delta^{\frac{N+1}{N}} \kappa(N)^{-\frac{N+1}{N}} (k+1)^{-\frac{N+1}{N}} \ ,
\end{align*}
where the last line follows by $j\geq 1$.

As for the first term in \ref{eqn:union_bound} we can also use Pinsker's inequality to get
\begin{align*}
\P (\exists t\in [N]:\ \hat{\mu}_t -\mu > z_t ) & \leq \exp \left( -\left( \frac{N+1}{N} \right)^2 \log \left( \kappa(N) \log_2 (2t) /\delta \right) \right) \\
& \leq \delta^{\frac{N+1}{N}} \kappa (N)^{-\frac{N+1}{N}} \sum_{t\in [N]} \log_2 (2t)^{-\frac{N+1}{N}} \ .
\end{align*}

The proof concludes the same way as that of Theorem~\ref{thm:anytime_1}, so that with the definition of $z_t$ above we have that
\[
\P (\exists t\in \N :\ \hat{\mu}_t -\mu > z_t ) \leq \delta \ .
\]


\section{The New Yorker Cartoon Caption Contest}\label{app:NYr}

\begin{figure}[h]\label{fig:NYr_example}
\begin{center}
\centering
\includegraphics[width = 4in]{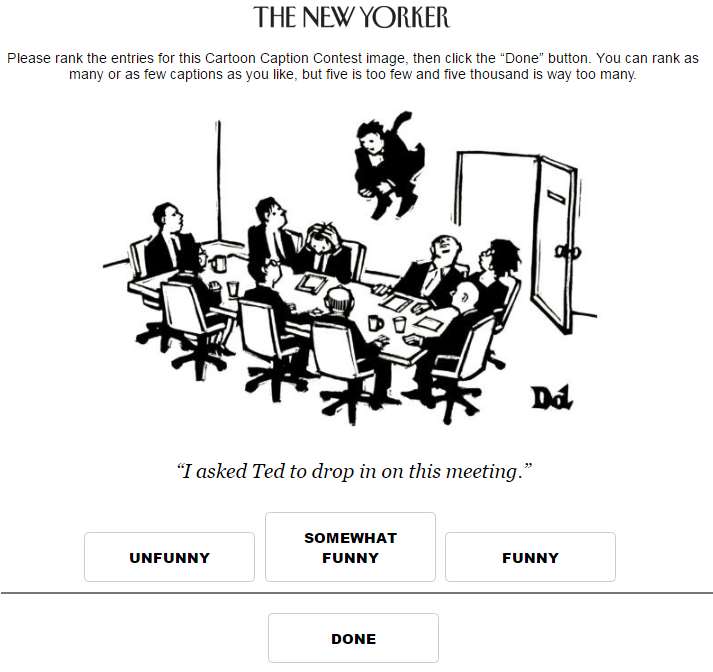}
\end{center}
\end{figure}

Each week a cartoon in need of a caption appears in The New Yorker magazine. The readers are invited to submit their ideas for funny captions to go with that cartoon. The New Yorker selects three finalists from the submissions, after which the readers select their favorite by voting online at \url{http://contest.newyorker.com/CaptionContest.aspx?tab=vote}.

\end{document}